\documentclass[12pt, a4paper]{article}
\usepackage[english]{babel}
\usepackage[iso-8859-7]{inputenc}
\usepackage{latexsym}
\usepackage{amsmath, amsthm, amssymb}
\usepackage{amssymb}
\usepackage{amsmath}
\usepackage{amsfonts}
\usepackage{amsthm}
\usepackage[all]{xy}
\usepackage{indentfirst}
\input{epsf.tex}

\newtheorem{definition}{Definition}[section]
\newtheorem{theorem}{Theorem}[section]
\newtheorem{proposition}{Proposition}[section]

\newtheorem{lemma}{Lemma}[section]

\newcommand{\nn}{\mathbb{N}}
\newcommand{\cc}{\mathbb{C}}

\newcommand{\rr}{\mathbb{R}}

\newcommand{\di}{\displaystyle}

\newcommand{\st}{\subset}
\newcommand{\sm}{\smallsetminus}

\begin{document}
\title{\bf Disjoint Hypercyclicity for families of Taylor-type Operators}

\author{V. Vlachou}         
\date{}
\maketitle 
\begin{abstract} 
We give necessary and sufficient condition so that we have d-hypercyclicity for  operators who map a holomorphic function to a partial sum of the Taylor expansion. This problem is connected with doubly universal Taylors series and this is an effort to generalize the concept to multiple universal Taylor series.
\end{abstract}
\noindent
\section{Introduction }
In the last 30 years, many authors have worked on the notion of hypercyclicity and  important advances in the research 
have been made, under several points of view. Roughly speaking, hypercyclicity means existence of a dense orbit. More recent papers have introduced and 
studied a new notion, the disjoint hypercyclicity i.e. 
the existence of a common vector with dense orbit for several operators, such
that the approximation of any fixed vectors is also simultaneously performed
by using a common subsequence. Our goal is to study disjoint hyperclycity for families of Taylor-type Operators.

Let us be more specific and  give the precise definition of  hypercyclicity (for more details see \cite{Bayart} and \cite{GE}).
\begin{definition}
Let $X,Y$ be  two topological vector spaces over $\mathbb{K}=\rr \text{ or } \cc$. A sequence of linear and continuous operatos $T_n: X\to Y, n=1,2,\ldots$ 
is said to be \textit{hypercyclic} if there exists a vector $x\in X$ so that the sequence
$$
\{ T_1x, T_2x,\ldots \},
$$
is dense in $Y$. In this case the vector $x$ will be called hypercyclic for $\{ T_n \}_{n\in\nn}$ and the symbol $HC(\{
T_n\}_{n\in\nn} )$ stands for the set of hypercyclic vectors for $\{ T_n\}_{n\in\nn}$. If the sequence $\{ T_n \}_{n\in\nn}$ comes from the
iterates of a single operator $T:X\to X$, i.e. $T_n=T^n$, $n=1,2,\ldots $ then $T$ is called hypercyclic and the set of
hypercyclic vectors for $T$ is denoted by $HC(T)$.
\end{definition}
We are now ready to give the definition of disjoint hypercyclicity as introduced in  \cite{bernal} and  \cite{bes}.
\begin{definition}\label{dh}
Let $\sigma_0\in\nn$ and $X$ , $Y_1, Y_{2},\ldots, Y_{\sigma_0}$ be topological vector spaces over $\mathbb{K}=\rr \text{ or } \cc$.  For each $\sigma\in\{1,2,\ldots,\sigma_0\}$ consider a sequence of linear and continuous operators $T_{\sigma,n}:X\to Y_{\sigma}$, $n=1,2,\ldots
$. We say that the sequences  $\{ T_{\sigma,n} \}_{n\in\nn}, \ \sigma=1,2,\ldots,\sigma_0$  are disjoint hypercyclic if the sequence  $[T_{1,n}, T_{2,n},\ldots, T_{\sigma_0,n}]:X\to Y_1\times Y_2\times\ldots Y_{\sigma_0}$ defined as:
$$ [T_{1,n}, T_{2,n},\ldots, T_{\sigma_0,n}](x)=(T_{1,n}(x), T_{2,n}(x),\ldots, T_{\sigma_0,n}(x))$$
is hypercyclic where $ Y_1\times Y_2\times\ldots Y_{\sigma_0}$ is assumed to be endowed with the product topology.
\end{definition}
The notion of d-hypercyclicity has been studied by many authors (see for example  \cite{bernal}- \cite{bes}) and it is a strong property which reflects in some sense
the density of the diagonal orbit.  Intresting  questions and problems have been studied in this setting and they have inspired G. Costakis and N.Tsirivas (see \cite{ct}) to
consider a similar question in the setting of universal Taylor series. We would like to continue along the same path of research (see also \cite{cv}).  

 So, let us   describe the specific operators that interest us.
We fix a simply connected domain $\Omega\st\cc$ and a point $\zeta_0\in\Omega$.
 We denote by $H(\Omega)$  the space
of functions, holomorphic in $\Omega$, endowed with the topology of uniform convergence on compacta. 
Moreover, for a compact set $K\subset\cc$, we denote
$$\mathcal{A}(K)=\{g\in H(K^o): \ g \text{ is
continuous on } K\}$$ 
 $$\mathcal{M}=\{K\st
\cc: K \ \text{compact set and } \ K^c \text{ connected
set}\}$$
and
 $$\mathcal{M}_{\Omega}=\{K\st
\cc\sm\Omega: K \ \text{compact set and } \ K^c \text{ connected
set}\}$$
For a function  $g$ defined on $K$, we use the notation $\di||g||_K=\sup_{z\in K} |g(z)|$.\\
Now for every  $K\in\mathcal{M}_{\Omega}$ and every sequence of natural numbers  $\{\lambda_{n}\}_{n\in\nn}$ we consider the sequence of operators:
$$T^{(\zeta_0)}_{\lambda_n}:H(\Omega)\to A(K), \ \ n=1,2,\ldots$$
$$T^{(\zeta_0)}_{\lambda_n}(f)(z)=\sum_{k=1}^{\lambda_{n}}\frac{f^{(k)}(\zeta_0)}{k!}(z-\zeta_0)^k,  \ \ n=1,2,\ldots.$$ 

V. Nestoridis  in \cite{N3} (see also \cite{N2}) proved that if the sequence $\{\lambda_n\}_{n\in\nn}$   is unbounded then the corresponding sequence of operators $\{T^{(\zeta_0)}_{\lambda_n}\}_{n\in\nn}$ is hypercyclic. 

In the first part of this work,  we consider a finite collection of sequences of operators of the above type and we study the problem of   disjoint hypercyclicity. This result generalizes the results in \cite{ct} and \cite{cv} on doubly universal Taylor series, where this porblem was investigated in the special case of two sequences of operators. Our tools include concepts and theorems from potential theory for which we would like to refer to \cite{R}. 
Lately, several authors have used potential theory in problems concerning universality (see \cite{cv}-\cite{ga4}, \cite{ma}, \cite{vmy1}-\cite{M-Y}, \cite{ts}).

In the second part, we deal with a special (finite) choice of sequences  of natural numbers and using Ostrowski-gaps we prove that the d-hyperciclic vectors are independent of the choice of $\zeta_0$. We use methodes and ideas used in \cite{GLM}, \cite{M-Y} (see also \cite{Luh} and \cite{N1}).


\section{D-Hypercyclicity for Taylor-type Operators}
\begin{definition} Let $\{\lambda^{(\sigma)}_n\}_{n\in\nn}$ , $\sigma=1,2,\ldots\sigma_0$ be a finite collection of sequences of natural numbers. A function $f\in H(\Omega)$ belongs to the class $U^{(\zeta_0)}_{mult}(\{\lambda^{(1)}_n\}_{n\in\nn}, \{\lambda^{(2)}_n\}_{n\in\nn},\ldots, \{\lambda^{(\sigma_0)}_n\}_{n\in\nn})$, if for every choice of compact sets $K_1,K_2,\ldots, K_{\sigma_0}\in\mathcal{M}_{\Omega}$  the set  
$$\{(T^{(\zeta_0)}_{\lambda^{(1)}_n}(f),T^{(\zeta_0)}_{\lambda^{(2)}_n}(f),\ldots,T^{(\zeta_0)}_{\lambda^{(\sigma_0)}_n}(f) )
: n\in\nn\}$$
is dense in $A(K_1)\times A(K_2)\times\ldots\times A(K_{\sigma_0})$.
\end{definition}
The main goal of  this section is to give necessary and suficient conditions  so that the above defined class of functions is non-empty.  Note that the functions of this class are disjoint hypercyclic vectors, for the sequences of operators we considered for every choice of compact sets $K_1,K_2,\ldots, K_{\sigma_0}\in\mathcal{M}_{\Omega}$.\\
\textbf{Remark:} The class  $U^{(\zeta_0)}_{mult}(\{\lambda^{(1)}_n\}_{n\in\nn}, \{\lambda^{(2)}_n\}_{n\in\nn},\ldots, \{\lambda^{(\sigma_0)}_n\}_{n\in\nn})$ is independent of the order with which we consider the sequences $\{\lambda^{(\sigma)}_n\}_{n\in\nn}$ , $\sigma=1,2,\ldots\sigma_0$.

 Nevertheless, in order to state our result we need to consider a specific arrangement for these sequences. 
\begin{definition}
Let $\{\lambda^{(\sigma)}_n\}_{n\in\nn}$, $\sigma=1,2,\ldots,\sigma_0$, $\sigma_0\in\nn$ be a finite number of sequences of natural numbers. We say that these sequences are well ordered if
$$\limsup_n\frac{\lambda^{(\sigma+1)}_n}{\lambda^{(\sigma)}_n}\geq \limsup_n\frac{\lambda^{(\sigma)}_n}{\lambda^{(\sigma+1)}_n}, \  \sigma=1,2,\dots,\sigma_0-1.$$
\end{definition}
\begin{lemma} Let $\{\lambda^{(\sigma)}_n\}_{n\in\nn}$, $\sigma=1,2,\ldots,\sigma_0$ be a finite number of sequences of natural numbers. There exists a rearrangement $\{\lambda^{(\pi(\sigma))}_n\}_{n\in\nn}$, $\sigma=1,2,\ldots,\sigma_0$ which is well ordered.
\end{lemma}
\begin{proof}
${}$\\
Step 1: If the sequences  $\{\lambda^{(1)}_n\}_{n\in\nn}$ and  $\{\lambda^{(2)}_n\}_{n\in\nn}$ satisfy the inequality
$$\limsup_n\frac{\lambda^{(2)}_n}{\lambda^{(1)}_n}\geq \limsup_n\frac{\lambda^{(1)}_n}{\lambda^{(2)}_n}$$
we take no action. If they do not satisfy the inequility we interchange their positions and then the inequility will be satisfied.\\
Step 2: Assume that the inequility is satisfied for $\sigma=1,2,\ldots,\sigma_1$, for some $\sigma_1\in\{1,\ldots,\sigma_0-2\}$. We will find a rearrengement so that the inequility is satisfied for  $\sigma=1,2,\ldots,\sigma_1+1$.
First we compare the sequences $\{\lambda^{(\sigma_1+1)}_n\}_{n\in\nn}$ and  $\{\lambda^{(\sigma_1+2)}_n\}_{n\in\nn}$ . If they also satisfy the inequility, we take no action and the result follows. If they do not satisfy the inequility we interchange them, so that the inequility is satisfied for $\sigma=\sigma_1+1$. Now we need to compare (the new) $\{\lambda^{\sigma_1+1}_n\}_{n\in\nn}$ with  $\{\lambda^{(\sigma_1)}_n\}_{n\in\nn}$.
If necessary we interchange them. In this case note that the inequility will hold for $\sigma=\sigma_1$ and it will still hold for 
$\sigma=\sigma_1+1$ because of our assumption. Continuing this way after a finite numbers of steps we will reach our goal.

Repeating the second step for $\sigma_1=1,2,\ldots,\sigma_0-2$ we will end up with a well ordered rearrangent.
\end{proof}
In view of the above, let us assume that we have a well ordered finite collection of sequences of natural numbers
 $\{\lambda^{(\sigma)}_n\}_{n\in\nn}$, $\sigma=1,2,\ldots,\sigma_0$, $\sigma_0\in\nn$.
 \begin{theorem}\label{main} The class  $U^{(\zeta_0)}_{mult}(\{\lambda^{(1)}_n\}_{n\in\nn}, \{\lambda^{(2)}_n\}_{n\in\nn},\ldots, \{\lambda^{(\sigma_0)}_n\}_{n\in\nn})$ is non-empty, if and only if, there exists a strictly increasing sequence of natural numbers $\{\mu_n\}_{n\in\nn}$ such that 
 $$\lim_{n\to\infty} \lambda^{(1)}_{\mu_n}=+\infty \text{ and } \lim_{n\to\infty}\frac{\lambda^{(\sigma+1)}_{\mu_n}}{\lambda^{(\sigma)}_{\mu_n}}=+\infty, \ \ \sigma=1,2,\ldots,\sigma_0-1.$$
 \end{theorem}
 First we will prove that the existence of such a sequence $\{\mu_n\}_{n\in\nn}$ implies that the class 
 $U^{(\zeta_0)}_{mult}(\{\lambda^{(1)}_n\}_{n\in\nn}, \{\lambda^{(2)}_n\}_{n\in\nn},\ldots, \{\lambda^{(\sigma_0)}_n\}_{n\in\nn})$ is $G_{\delta}$ and dense subset of $H(\Omega)$.  For this task we need a proposition, which is a modification of the well known theorem of Bernstein-Walsh (theorem 6.3.1  \cite{R}, see also \cite{ct} and \cite{cv}). We would like to note that this idea was also used in  \cite{ct} and \cite{cv}, but the corresponding propositions were not enough for $\sigma_0> 2$. Therefore this proposition is actually the key to obtain the result for more sequences. \\
 To state our proposition in a simple way, we first give a definition.
\begin{definition} Let $h_n:U\to \cc , \ n=1,2,\ldots$ be a sequence of continuous functions   defined on an open  set $U$ and $\sigma_n, \ n=1,2,\ldots$ be a sequence of positive integers.
We say that the sequence  $h_n, \ n=1,2\ldots$ is $\{\sigma_n\}-$locally bounded if for every compact set $K\st U$ the sequence
$||h_n||^{\frac{1}{\sigma_n}}_{K}$ is bounded.
\end{definition}
 \begin{proposition}\label{bw} Let $K\in\mathcal{M}$. For every $\{\sigma_n\}-$locally bounded sequence $\{f_n\}_{n\in\nn}$ of holomorhic functions 
 on an  open neighbourhood $U$ of $K$:
 $$\limsup_n d_{\tau_n}(f_n,K)^{\frac{1}{\tau_n}}\leq\theta<1$$
 where
$$\theta=\begin{cases}\sup_{\cc_\infty\smallsetminus U}exp(-g_{\cc_\infty\smallsetminus K}(z,\infty)), \  &\text{ if }
c(K)>0, \\ 0,  \  &\text{ if }
c(K)=0\end{cases}$$
and $\{\tau_n\}_{n\in\nn}$ is any  sequence of natural numbers such that $\di\lim_n\frac{\tau_n}{\sigma_n}=+\infty$.
  \end{proposition}
  \begin{proof}Assume first that $c(K)>0$.  Following the proof of theorem 6.3.1  in \cite{R}, we consider a closed contour $\Gamma$ in $U\smallsetminus  K$ such that
  $ind_{\Gamma}(z)=1, \ z\in  K$ and  $ind_{\Gamma}(z)=0, \ z\notin U$.
 
Since $\sigma_n\geq 1,n=1,2,\dots$,  $\di\lim_{n\to\infty}\tau_n=+\infty,$  so for $n$ large enough $\tau_n\geq 2$.  
In this case we may consider a  Fekete polynomial $q_{\tau_n}$ of degree $\tau_n$ for $K$ and we define 
 $$p_n(w)=\frac{1}{2\pi i}\int_{\Gamma}\frac{f_n(z)}{q_{\tau_n}(z)}\cdot\frac{q_{\tau_n}(w)-q_{\tau_n}(z)}{w-z}dz, \ w\in K.$$
Then (as in the proof in \cite{R}) $p_n$ is a polynomial of degree at most ${\tau_n-1}$.
Moreover, using Cauchy's integral formula we conclude that:
$$f_n(w)-p_n(w)=\frac{1}{2\pi i}\int_{\Gamma}\frac{f_n(z)}{w-z}\cdot\frac{-q_{\tau_n}(w)}{q_{\tau_n}(z)}dz, \ w\in K$$
Thus,
\begin{equation}\label{co}||f_n-p_n||_{K}\leq \frac{1}{2\pi} \cdot \ell(\Gamma)\cdot \frac{1}{dist(\Gamma, K)}\cdot \frac{||q_{\tau_n}||_K}{\di\min_{z\in\Gamma}|q_{\tau_n}(z)|}\cdot||f_n||_{\Gamma}, \end{equation}
where $\ell(\Gamma)$ is the length of $\Gamma$ and $dist(\Gamma, K)$ is the distance of 
$\Gamma$ from $K$.

Since $f_n$ is $\{\sigma_n\}-$ locally bounded, there exists a positive constant $A>1$ such that $||f_n||_{\Gamma}\leq A^{\sigma_n}$. 

Furthemore
in the proof of theorem 6.3.1  in \cite{R}, it is proved that:
$$\limsup_n \biggl(\frac{||q_{\tau_n}||_K}{\di\min_{z\in\Gamma}|q_{\tau_n}(z)|}\biggl)^{\frac{1}{\tau_n}}\leq \alpha,$$
with $\di\alpha= \sup_{z\in \Gamma}\exp(-g_{\cc_\infty\smallsetminus K}(z,\infty))$.

Thus,
$$\limsup_n  d_{\tau_n}(f_n,K)^{\frac{1}{\tau_n}} \leq  \limsup_n ||f_n-p_n||_{K}^{\frac{1}{\tau_n}}\leq \alpha.$$
(note that $\lim_n C^{\frac{1}{\tau_n}}=1$ and  $\lim_n A^{\frac{\sigma_n}{\tau_n}}=1$).
The rest of the proof is exactly the same as in theorem 6.3.1  in \cite{R}.
  \end{proof}
  \begin{theorem}\label{m}If $\di\lim_n\lambda_n^{(1)}=+\infty$  and $\di\lim_{n\to\infty}\frac{\lambda^{(\sigma+1)}_{n}}{\lambda^{(\sigma)}_{n}}=+\infty, \ \sigma=1,2,\ldots,\sigma_0-1$, then the class  $U^{(\zeta_0)}_{mult}(\{\lambda^{(1)}_n\}_{n\in\nn}, \{\lambda^{(2)}_n\}_{n\in\nn},\ldots, \{\lambda^{(\sigma_0)}_n\}_{n\in\nn})$ is a $G_{\delta}$ and dense subset of $H(\Omega)$.
  \end{theorem}
  \begin{proof}
  Let $\{f_j\}_{j\in\nn}$ be an enumeration of polynomials with rational coefficients. Let, in addition,   $\{K_m\}_{m\in\nn}$ be a sequence of compact sets in $\mathcal{M}_{\Omega}$, such that the following holds: every  $K\in \mathcal{M}_{\Omega}$,  is contained in some $K_{m}$
   (for the existence of such a sequence we refer to\ \cite{N3}).\\
 For every choice of positive integers $s$, $n$, $m_{\sigma},  \ \sigma=1,2,\ldots,\sigma_0$ and  \\$j_{\sigma}, \ \sigma=1,2,\ldots,\sigma_0$, we set:
$$E(\{m_{\sigma}\}_{\sigma=1}^{\sigma_0}, \{j_{\sigma}\}_{\sigma=1}^{\sigma_0},s,n)=
\{f\in H(\Omega): ||T^{(\zeta_0)}_{\lambda_n^{(\sigma)}}-f_{j_{\sigma}}||_{K_{\sigma}}<\frac{1}{s}, \ \sigma=1,2,\ldots,\sigma_0\}$$  
In view of Mergelyan's theorem, it is easy to see that 
 $$U^{(\zeta_0)}_{mult}(\{\lambda^{(1)}_n\}_{n\in\nn}, \{\lambda^{(2)}_n\}_{n\in\nn},\ldots, \{\lambda^{(\sigma_0)}_n\}_{n\in\nn})=$$
 $$=\bigcap_{\{m_{\sigma}\}_{\sigma=1}^{\sigma_0}}\bigcap_{\{j_{\sigma}\}_{\sigma=1}^{\sigma_0}}\bigcap_{s}\bigcup_{n}\\
 E(\{m_{\sigma}\}_{\sigma=1}^{\sigma_0}, \{j_{\sigma}\}_{\sigma=1}^{\sigma_0},s,n).$$
Hence, in view of Baire's Category Theorem, it suffices to prove  that\\ $\di\bigcup_{n}E(\{m_{\sigma}\}_{\sigma=1}^{\sigma_0}\{j_{\sigma}\}_{\sigma=1}^{\sigma_0},s,n)$ is dense in $H(\Omega)$.
(see also proposition 2.3 in \cite{bes}).\\
For this reason we fix $g\in H(\Omega)$, $\varepsilon>0$, and $L\subset \Omega$ compact. Without loss of generality,
we may assume that $L$ has connected complement (note that $\Omega$ is simply connected), $\zeta_0\in L^o$ (if not we work with a larger $L$) and $\lambda_n^{(\sigma+1)}>\lambda_n^{(\sigma)}, \ \ \sigma=1,2,\ldots, \sigma_0-1$
(this holds for $n$ large enough).

In view of Runge's theorem, we may fix a polynomial $p$ such that:
$$||g-p||_{L}<\frac{\varepsilon}{2} \text{ and } ||p-f_{j_1}||_{K_{m_1}}<\frac{1}{s}.$$
Fix  two open and disjoint sets $U_1, U_2$ with
 $L\subset U_1$ and $\cup_{\sigma=1}^{\sigma_0}K_{m_{\sigma}} \subset U_2$.
 
For every $\sigma=2,\ldots,\sigma_0$,  we will construct via a finite induction a sequence of polynomials $\{Q_{n}^{(\sigma)}\}_{n\in\nn}$ with the following properties:\\
$\bullet$ The degree of the  terms of $Q_{n}^{(\sigma)}$ varies between $\lambda_n^{(\sigma-1)}+1$ and $\lambda_n^{(\sigma)}$.\\
$\bullet$ $||Q_n^{(\sigma)}(z-\zeta_0)||_L\xrightarrow{n\to\infty} 0$.\\
$\bullet$ $ ||p(z)+\sum_{k=2}^{\sigma}Q_n^{(k)}(z-\zeta_0)-f_{j_{\sigma}}(z)||_{K_{m_{\sigma}}}\xrightarrow{n\to\infty} 0$.

Let  $\sigma\in\{2,\ldots,\sigma_0\}$. If $\sigma\geq 3$
 assume, in addition, that   the previous  sequences of polynomials have been defined.

We apply proposition \ref{bw} for\\
 $U=(U_1-\zeta_0)\cup(U_2-\zeta_0)$,\\
  $K=(L-\zeta_0)\cup (K_{m_{\sigma}}-\zeta_0)$,\\
$f_n(z)=\begin{cases}z^{-\lambda_n^{(\sigma-1)}-1}g_n(z), &z\in U_2-\zeta_0, \\ 0, &z\in U_1-\zeta_0\end{cases}$,\\
where $g_n(z)= f_{j_{\sigma}}(z+\zeta_0)-p(z+\zeta_0)-\sum_{k=2}^{\sigma-1}Q_n^{(k)}(z)  $\\ $\{\sigma_n\}=\{\lambda_n^{(\sigma-1)}+1\}$ and $\{\tau_n\}=\{\lambda_n^{(\sigma)}-(\lambda_n^{(\sigma-1)}+1)\}$.\\
Note that in case $\sigma=2$ we need to set $g_n(z)= f_{j_{2}}(z+\zeta_0)-p(z+\zeta_0)$.\\
Let us stress out why the sequence of functions $\{f_n\}_{n\in\nn}$ is $\{\sigma_n\}-$ locally bounded.
We will deal with the case $\sigma>2$. \\
Let $\tilde{K}\st U$ be a compact set. Since $f_n$ are zero on $U_1-\zeta_0$ we may assume that $\tilde{K}\st (U_2-\zeta_0)$. 

 For every $n$,  the function  $\sum_{k=2}^{\sigma-1}Q_n^{(k)}(z)$ is a polynomial of degree at most 
$\lambda_n^{(\sigma-1)}$.

Our assumption implies that $||Q_n^{(k)}(z-\zeta_0)||_L\to 0$, therefore for $n$ large enough 
$$ ||\sum_{k=2}^{\sigma-1}Q_n^{(k)}(z)||_{L-\zeta_0}<1.       $$
In view of Bernstein's Lemma (a) (see \cite{R} p.156),  if $d_n$ is the degree of $\sum_{k=2}^{\sigma-1}Q_n^{(k)}(z)$
we have:
$$\biggl|\sum_{k=2}^{\sigma-1}Q_n^{(k)}(z)\biggl|^{\frac{1}{d_n}}\leq e^{g_D(z,\infty)}\biggl|\biggl|\sum_{k=2}^{\sigma-1}Q_n^{(k)}(z)\biggl|\biggl|^{\frac{1}{d_n}}_{L-\zeta_0}<e^{g_D(z,\infty)},$$ 
for $D=\cc_{\infty}-(L-\zeta_0)$ and $z\in D\smallsetminus\{\infty\}$. The compact set $L-\zeta_0$ is non-polar since it contains an open disk of center 0. The function $e^{g_D(z,\infty)}$ is bounded and continuous on $\tilde{K}$. Thus we may choose 
$A=\di\max_{z\in\tilde{K}}|e^{g_D(z,\infty)}|+1$. Then:
$$ \biggl|\biggl|\sum_{k=2}^{\sigma-1}Q_n^{(k)}(z)\biggl|\biggl|_{\tilde{K}}^{\frac{1}{d_n}} < A\Rightarrow
 \biggl|\biggl|\sum_{k=2}^{\sigma-1}Q_n^{(k)}(z)\biggl|\biggl|_{\tilde{K}}<A^{d_n}\leq A^{\lambda_n^{(\sigma-1)}+1}=A^{\sigma_n}$$
 We are ready to return to the functions $f_n$:
 $$||f_n||_{\tilde{K}}\leq \biggl(\max_{z\in\tilde{K}}\frac{1}{|z|}\biggl)^{\sigma_n}(C+A^{\sigma_n}),$$
where $C=||f_{j_{\sigma}}-p||_{\tilde{K}+\zeta_0}$ and there result follows. \\
The last argument suffices for the case 
$\sigma=2$
as well (set $A=0$.)
\\
Since all the requirements of the proposition \ref{bw}  are fulfilled 
we conclude that:
$$\limsup_n d_{\tau_n}(f_n,K)^{\frac{1}{\tau_n}}\leq \theta<1,\ \text{ for a suitable } \theta<1.$$ 
Hence if we fix $\theta_0\in (\theta,1)$, there exists $n_0\in \nn$ with:
$$ d_{\tau_n}(f_n,K)^{\frac{1}{\tau_n}}<\theta_0, \ \ n\geq n_0. $$
It is now apparent that we can fix a sequence of polynomials $p_n$ with degree less or equal to $\tau_n$ such that:
\begin{equation}\label{theta}
||f_n-p_n||_{K}<\theta_0^{\tau_n}, \ n\geq n_0.
\end{equation}
We set $Q_n^{(\sigma)}(z)=z^{\lambda_n^{(\sigma-1)}+1}\cdot p_n(z), \ n\in\nn$.\\
Obviously, the degree of the terms of $Q_n^{(\sigma)}$ varies between $ \lambda_n^{(\sigma-1)}+1$ and $\lambda_n^{(\sigma)}$, so the first requirement is satisfied.\\
For the second requirement we set $M=||z-\zeta_0||_{L\cup K_{m_{\sigma}}}+1$ and we have:
$$||Q_n^{(\sigma)}(z-\zeta_0)||_{L}\leq M^{\lambda_n^{(\sigma-1)}+1}\cdot ||p_n||_{L-\zeta_0}\leq$$
$$\leq M^{\lambda_n^{(\sigma-1)}+1}||p_n-f_n||_{K}<M^{\lambda_n^{(\sigma-1)}+1}\theta_0^{\lambda_n^{(\sigma)}},$$
where we have used relation (\ref{theta}).

It is easy to see that $||Q_n^{(\sigma)}(z-\zeta_0)||_{L}\to 0$.\\
We are ready to proceed to the third requirement:
$$||p(z)+\sum_{k=2}^{\sigma}Q_n^{(k)}(z-\zeta_0)-f_{j_{\sigma}}(z)||_{K_{m_{\sigma}}}=$$
$$=||f_{j_{\sigma}}(z+\zeta_0)-p(z+\zeta_0)-\sum_{k=2}^{\sigma-1}Q_n^{(k)}(z)-Q_n^{(\sigma)}(z)||_{K_{m_{\sigma}}-\zeta_0}=$$
$$=||z^{\lambda_n^{(\sigma-1)}+1}(f_n(z)-p_n(z)||_{K_{m_{\sigma}}-\zeta_0}\leq M^{\lambda_n^{(\sigma-1)}+1}||f_n-p_n||_{K}<M^{\lambda_n^{(\sigma-1)}+1}\theta_0^{\lambda_n^{(\sigma)}},$$ so as before:
$$||p(z)+\sum_{k=2}^{\sigma}Q_n^{(k)}(z-\zeta_0)-f_{j_{\sigma}}(z)||_{K_{m_{\sigma}}}\xrightarrow{n\to\infty}0.$$
To finish the proof, we claim that the function $$f(z)=p(z)+\sum_{k=2}^{\sigma_0}Q_{n_1}^{(k)}(z-\zeta_0),$$ for a suitable choice of $n_1\in\nn$ is near $g$ on $L$ and belongs to the set $E(\{m_{\sigma}\}_{\sigma=1}^{\sigma_0}, \{j_{\sigma}\}_{\sigma=1}^{\sigma_0},s,n_1)$.

 Let us see why:\\
Since $||\sum_{k=2}^{\sigma_0}Q_{n}^{(k)}(z-\zeta_0)||_L\to 0$, for $n_1$ large enough 
 $$||f-g||_L\leq ||p-g||_L+ ||\sum_{k=2}^{\sigma_0}Q_{n_1}^{(k)}(z-\zeta_0)||_L< 2||p-g||_L<\varepsilon.$$
 Moreover for $\sigma=1$
  it suffices to have $\lambda^{(1)}_{n_1}>deg p$, because then $T^{(\zeta_0)}_{\lambda_n^{(1)}}(f)=p$ so
$$||T^{(\zeta_0)}_{\lambda_n^{(1)}}(f)-f_{j_1}||_{K_{m_{1}}}=||p-f_{j_1}||_{K_{m_{1}}}<\frac{1}{s}.$$
 and  for $\sigma\geq2$:
 $$||T^{(\zeta_0)}_{\lambda_n^{(\sigma)}}(f)-f_{j_{\sigma}}||_{K_{m_{\sigma}}}=||p(z)+\sum_{k=2}^{\sigma}Q_{n_1}^{(k)}(z-\zeta_0)-f_{j_{\sigma}}(z)||_{K_{m_{\sigma}}}$$
and for $n_1$ large enough, it is less than $\di\frac{1}{s}$.
     \end{proof}
     We are now ready to prove that otherwise the class is empty.\\
     Let us start with a lemma (see also \cite{vmy1}).
      \begin{lemma}\label{sets} Let $\Omega\st \cc$ be a simply connected domain. Then their exists an increasing sequence of compact sets $E_k, \ k=1,2,\ldots$ with 
     the following properties:\\
     (i)  $E_k\in\mathcal{M}_{\Omega}, \ k=1,2,\ldots$\\
   (ii) $\bigcup_{k}E_k$ is closed and non-thin at $\infty$.
     \end{lemma}
     \begin{proof}
       If $\Omega$ is not bounded we set $E_k=\Omega^c\cap \overline{D(\zeta_0,k)}, \ k\in \nn.$ Then the sets  $E_k$ belong to $\mathcal{M}$, they are disjoint from $\Omega$ and their union  is closed and non-thin at $\infty$.
 ( Note that $\di\bigcup_{k\in\nn}E_k= \Omega^c$  is connected and contains more than one points, so this follows from Theorem 3.8.3 p. 79 \cite{R}.)\\
 If, on the other hand $\Omega$ is bounded, fix $N\in\nn$ with $\Omega\st D(0,N)$ and set $E_k=[N, N+k]$, $k\in\nn$.
 Again $E_k\in\mathcal{M}$, they are disjoint from $\Omega$ and $\di\bigcup_{k\in\nn}E_k=[N,+\infty)$ is closed and non-thin at $\infty$.  In both case the sequence of sets $E_k$ is increasing.\\
 This completes the proof.\\
     \end{proof}
     \textit{ Proof of Theorem \ref{main}       }
 If there exists such a sequence $\{\mu_n\}_{n\in\nn}$, then in view of Theorem \ref{m} the class
  $U^{(\zeta_0)}_{mult}(\{\lambda^{(1)}_{n}\}_{n\in\nn}, \{\lambda^{(2)}_{n}\}_{n\in\nn},\ldots, \{\lambda^{(\sigma_0)}_{n}\}_{n\in\nn})$ is a $G_{\delta}$ and dense subset of $H(\Omega)$.\\
  Now, let us assume that there exists no such sequence. \\
We argue by a contrudiction and we assume that there exists  a function $$f\in U^{(\zeta_0)}_{mult}(\{\lambda^{(1)}_{n}\}_{n\in\nn}, \{\lambda^{(2)}_{n}\}_{n\in\nn},\ldots, \{\lambda^{(\sigma_0)}_{n}\}_{n\in\nn})$$ \\
In view of Lemma \ref{sets}, we may fix a sequence of sets $\{E_k\}$ as stated in the lemma.
As a result we may fix a strictly increasing sequence of natural numbers $\{n_k\}_{n\in\nn}$ such that the following holds:
\begin{equation}\label{odd}
||T^{(\zeta_0)}_{\lambda_{n_k}^{(\sigma)}}(f)||_{E_k}<\frac{1}{k}, \ \ \sigma\in\{1,2,\ldots,\sigma_0\} \ \text{ odd}.
\end{equation}
\begin{equation}\label{even}||T^{(\zeta_0)}_{\lambda_{n_k}^{(\sigma)}}(f)-1||_{E_k}<\frac{1}{k}, \ \ \sigma\in\{1,2,\ldots,\sigma_0\} \ \text{ even}.
\end{equation}
\textbf{Remark:} We may also choose $\{\lambda^{(\sigma)}_{n_k}\}_{k\in\nn}$ to be striclty increasing for every $\sigma=1,2,\ldots,\sigma_0$. (this is well known and has been stated often in articles on Universal Taylor Series see for example \cite{N1}). Thus we have
$\di\lim_{n\to\infty}\lambda^{(\sigma)}_{n_k}=+\infty$, for every $\sigma=1,2,\ldots,\sigma_0$.\\
\textbf{ Case I: } $\di \limsup_{k\to\infty}\frac{\lambda_{n_k}^{(2)}}{\lambda_{n_k}^{(1)}}<+\infty$.\\
We have assumed that the sequences are well-ordered, thus 
$$\di \limsup_{k\to\infty}\frac{\lambda_{n_k}^{(2)}}{\lambda_{n_k}^{(1)}}\geq\di \limsup_{k\to\infty}\frac{\lambda_{n_k}^{(1)}}{\lambda_{n_k}^{(2)}}.$$
Therefore, we may fix a positive number $C>0$ with:
$$\frac{\lambda_{n_k}^{(2)}}{\lambda_{n_k}^{(1)}}<C \text{ \ \ and \ \  }\frac{\lambda_{n_k}^{(1)}}{\lambda_{n_k}^{(2)}}<C, \ k\in\nn.$$
We consider two sets of natural numbers:
$$I=\{k\in\nn: \lambda_{n_k}^{(2)}\geq \lambda_{n_k}^{(1)}\}$$
$$J=\{k\in\nn: \lambda_{n_k}^{(1)}\geq \lambda_{n_k}^{(2)}\}.$$
At least one of the above sets is infinite. Lets us assume first that $I$ is infinite.
We set:
$$p_k(z)=\biggl(\frac{R}{z-\zeta_0}\biggl)^{\lambda_{n_k}^{(1)}}\biggl(T^{(\zeta_0)}_{\lambda_{n_k}^{(2)}}(f)(z)-T^{(\zeta_0)}_{\lambda_{n_k}^{(1)}}(f)(z)\biggl), \ k\in I$$
where $R=dist(\Omega^c,\zeta_0)>0$.\\
Then $p_k$ are polynomials and $deg p_k\leq \lambda_{n_k}^{(2)}-\lambda_{n_k}^{(1)}=\lambda_{n_k}^{(1)}\biggl(\frac{\lambda_{n_k}^{(2)}}{\lambda_{n_k}^{(1)}}-1\biggl)<C \lambda_{n_k}^{(1)}$.\\
Set $E=\biggl(\di\bigcup_{k\in\nn}E_k\biggl)\cap D(\zeta_0,2R)^c$.  Then $E$ is closed and non-thin at $\infty$\\
(note that non-thiness is a local property see p. 79 in \cite{R}).\\
Let $z\in E$. Then $z\in E_k$, $k$ large enough and $|z-\zeta_0|\geq 2R$. Thus for $k\in I$ large enough we have:
$$|p_k(z)|\leq \biggl|\frac{R}{z-\zeta_0}\biggl|^{\lambda_{n_k}^{(1)}}\cdot \biggl(||T^{(\zeta_0)}_{\lambda_{n_k}^{(2)}}||_{E_k}+  ||T^{(\zeta_0)}_{\lambda_{n_k}^{(1)}}||_{E_k}       \biggl)
\leq  \biggl(\frac{1}{2}\biggl)^{\lambda_{n_k}^{(1)}}(1+\frac{2}{k})<3 \biggl(\frac{1}{2}\biggl)^{\lambda_{n_k}^{(1)}}$$
(we have used relations (\ref{odd}) and (\ref{even}).)\\
Thus:
$$\limsup_{k\in I}|p_k(z)|^{\frac{1}{C\lambda^{(1)}_{n_k}}}\leq \biggl(\frac{1}{2}\biggl)^{\frac{1}{C}}<1, \ \ z\in E.$$
Moreover, if $\Gamma\st E$ is a continuum (compact, connected but not a singleton) we have:
$$\limsup_{k\in I}||p_k||_{\Gamma}^{\frac{1}{C\lambda^{(1)}_{n_k}}}\leq \biggl(\frac{1}{2}\biggl)^{\frac{1}{C}}<1.$$
Therefore, in view of Theorem 1 in  \cite{M-Y}, we conclude that $p_k\to 0, \ k\in I$ compactly on $\cc$.\\
Let $\xi\in\partial\Omega$ with $|\xi-\zeta_0|=R$. Then from the above
 $$\biggl(\frac{\xi-\zeta_0}{R}\biggl)^{\lambda_{n_k}^{(1)}}p_k(\xi)\to 0, \ k\in I$$
 But  $$\biggl(\frac{\xi-\zeta_0}{R}\biggl)^{\lambda_{n_k}^{(1)}}p_k(\xi)=T^{(\zeta_0)}_{\lambda_{n_k}^{(2)}}(f)(\xi)-T^{(\zeta_0)}_{\lambda_{n_k}^{(1)}}(f)(\xi).$$
 Thus,
 $$\biggl|\biggl(\frac{\xi-\zeta_0}{R}\biggl)^{\lambda_{n_k}^{(1)}}p_k(\xi)-1\biggl|\leq  ||T^{(\zeta_0)}_{\lambda_{n_k}^{(2)}}(f)-1||_{E_k}+||T^{(\zeta_0)}_{\lambda_{n_k}^{(1)}}(f)||_{E_k}\leq\frac{2}{k}\to 0.$$
 So we have arrived to a contrudiction.\\
 Now if $J$ is infinite, we set 
 $$p_k(z)=\biggl(\frac{R}{z-\zeta_0}\biggl)^{\lambda_{n_k}^{(2)}}\biggl(T^{(\zeta_0)}_{\lambda_{n_k}^{(1)}}(f)(z)-T^{(\zeta_0)}_{\lambda_{n_k}^{(2)}}(f)(z)\biggl), \ k\in J$$
 and following the same arguments again we arrive to a contrudiction.\\
 \textbf{ Case 2: } $\di \limsup_{k\to\infty}\frac{\lambda_{n_k}^{(2)}}{\lambda_{n_k}^{(1)}}=+\infty$. Then passing to a subsequence we may assume that
  $\di \lim_{k\to\infty}\frac{\lambda_{n_k}^{(2)}}{\lambda_{n_k}^{(1)}}=+\infty$.
Now  if $\di \limsup_{k\to\infty}\frac{\lambda_{n_k}^{(3)}}{\lambda_{n_k}^{(2)}}<+\infty$ we arrive to a contrudiction as in case 1. Therefore we conclude that 
 $\di \limsup_{k\to\infty}\frac{\lambda_{n_k}^{(3)}}{\lambda_{n_k}^{(2)}}=+\infty$, so passing to a subsequence we may assume that  $\di \lim_{k\to\infty}\frac{\lambda_{n_k}^{(3)}}{\lambda_{n_k}^{(2)}}=+\infty$.
 Continuing this way after a finite number of steps we will end up with a sequence $\{\mu_n\}_{n\in\nn}$ that we assumed that it does not exist.
 The proof of the theorem is complete.
\section{Independance of choice of expansion}
We start by giving the definition of Ostrowski-gaps, since they will play a central role in this section.
\begin{definition} Let $\di\sum_{k=0}^{\infty}a_{k} (z-\zeta_0)^k$ be a power series with positive radious of convergence. We say that it has Ostrowski gaps $(p_m,q_m), \ m=1,2,\ldots$, if there
exist two sequences of natural numbers $\{p_m\}_{m\in\nn}$ and $\{q_m\}_{m\in\nn}$ such that the following hold:\\
(i) $p_1<q_1\leq p_2<q_2\leq\ldots$ and $\di\lim_{m}\frac{q_m}{p_m}=\infty$\\
(ii)For $I=\cup_{m=1}^{\infty}\{p_m+1,\ldots,q_m\}$ we have $\di\lim_{\nu\in I}|a_{\nu}|^{\frac{1}{\nu}}=0.$
\end{definition}
\begin{theorem} The class $U^{(\zeta_0)}_{mult}(\{n\}_{n\in\nn}, \{n^2\}_{n\in\nn},\ldots,                   \{n^{\sigma_0}\}_{n\in\nn})$ is indepedant of the choice of $\zeta_0$.
\end{theorem}
\begin{proof}Let $f\in U^{(\zeta_0)}_{mult}(\{n\}_{n\in\nn}, \{n^2\}_{n\in\nn},\ldots,                   \{n^{\sigma_0}\}_{n\in\nn})$.  Let $K_1,\ldots,K_{\sigma_0}\in\mathcal{M}_{\Omega}$, $g_1\in A(K_1),\ldots,g_{\sigma_0}\in A(K_{\sigma})$ and $L\st\Omega$ compact.
Fix a sequence $\{E_k\}_{k\in\nn}$ as in lemma \ref{sets} with the additional property that every  $E_k$ disjoint from $\bigcup_{\sigma}K_{\sigma}$ and set $E=\cup_{k}E_k$.
Then, there exists a strictly increasing sequence $\{n_k\}_{k\in\nn}$ with:
$$ ||T^{(\zeta_0)}_{n_k^{\sigma}}(f)-g_{\sigma}||_{K_{\sigma}}\xrightarrow{n\to\infty}0, \ \ \text{ for every } \sigma=1,2,\ldots,\sigma_0                    $$
$$   ||T^{(\zeta_0)}_{n_k^{\sigma}}(f)||_{E_k}\xrightarrow{n\to\infty}0, \ \ \text{ for every } \sigma=1,2,\ldots,\sigma_0                                    $$
Note that the functions $T^{(\zeta_0)}_{n_k^{\sigma}}(f)$ are polynomials of degree less or equal to $n_k^{\sigma}$. 
Moreover, for $k$ large enough 
$$ ||T^{(\zeta_0)}_{n_k^{\sigma}}(f)||_{E_k}\leq 1\Rightarrow ||T^{(\zeta_0)}_{n_k^{\sigma}}(f)||^{\frac{1}{n_k^{\sigma}}}_{E_k}\leq 1
\Rightarrow \limsup_{k}|T^{(\zeta_0)}_{n_k^{\sigma}}(f)(z)|^{\frac{1}{n_k^{\sigma}}}\leq 1, \ \ z\in E.$$
In view of lemma 2 in \cite{M-Y} 
$$ \limsup_{k}||T^{(\zeta_0)}_{n_k^{\sigma}}(f)||^{\frac{1}{n_k^{\sigma}}}_{|z|=R}\leq 1, \ \ \forall R>0.$$
Passing to a subsequence, we may assume that:
$$||T^{(\zeta_0)}_{n_k^{\sigma}}(f)||^{\frac{1}{n^{\sigma}_k}}_{|z|=k}\leq 2, \ \sigma=1,2,\ldots,\sigma_0$$
Now set:
$$p_k=\biggl[\frac{n_k}{(\log k)^{\frac{1}{\sigma_0}}}\biggl]+1.$$
Then
$$   \frac{1}{(\log k)^{\frac{1}{\sigma_0}}}     \leq \frac{p_k}{n_k}\leq    \frac{1}{(\log k)^{\frac{1}{\sigma_0}}}+\frac{1}{n_k}  $$
Thus,
$\di \biggl( \frac{n_k}{p_k}\biggl)^{\sigma} \leq \log k $ and $\di \biggl( \frac{n_k}{p_k}\biggl)^{\sigma}\to +\infty$ for all $\sigma=1,2,\ldots,\sigma_0$.\\
Moreover, if $p_k^{\sigma}\leq \nu\leq n^{\sigma}_k$ we have:
$$|a_{\nu}|^{\frac{1}{\nu}}\leq \frac{||T^{(\zeta_0)}_{n_k^{\sigma}}(f)||^{\frac{1}{\nu}}_{|z|=k}}{k}=\frac{\biggl(||T^{(\zeta_0)}_{n_k^{\sigma}}(f)||^{\frac{1}{n^{\sigma}_k}}_{|z|=k}\biggl)^{\frac{n^{\sigma}_k}{\nu}}}{k} \leq
\frac{2^{(\frac{n_k}{p_k})^{\sigma}}}{k}\leq \frac{2^{\log k}}{k}=k^{\log 2-1}$$
Thus $\lim_{\nu}|a_{\nu}|^{\frac{1}{\nu}}=0$ and the power series has Ostrowski-gaps $(p_k^{\sigma}, n_k^{\sigma})$, \\ $ k=1,2,\ldots$ for every $\sigma=1,2,\ldots,\sigma_0$.\\
It is known (see \cite{Luh}) that in this case:
$$T^{(\zeta_0)}_{p_k^{\sigma}}(f)- T^{(\zeta_0)}_{n_k^{\sigma}}(f)\xrightarrow{k\to \infty} 0, \text{ compactly on } \cc.$$
Moover, in view of lemma 9.2 \cite{N1} (see also theorem 1 \cite{Luh}) we have:
$$\sup_{\zeta\in L}\sup_{z\in K}|T^{(\zeta_0)}_{p_k^{\sigma}}(f)(z)-T^{(\zeta)}_{p_k^{\sigma}}(f)(z)|\xrightarrow{k\to \infty} 0,$$
for every choice of compact sets $L\st\Omega$ and $K\st \cc$.\\
Thus:
$$\sup_{\zeta\in L} ||T^{(\zeta)}_{n_k^{\sigma}}(f)-g_{\sigma}||_{K_{\sigma}}\xrightarrow{n\to\infty}0, \ \ \text{ for every } \sigma=1,2,\ldots,\sigma_0 .$$
So $f\in  U^{(\zeta)}_{mult}(\{n\}_{n\in\nn}, \{n^2\}_{n\in\nn},\ldots,                   \{n^{\sigma_0}\}_{n\in\nn})$ for every $\zeta\in\Omega$ and the result follows.

\end{proof}
\textbf{Remark:} In this case we have d-hypercyclicity for uncountable many sequences of operators $\{T^{(\zeta)}_{n^{\sigma}}\}_{n\in\nn},$ $\sigma=1,2,\ldots,\sigma_0$ and $\zeta\in \Omega$.

 V.Vlachou, \\
Department of Mathematics,\\
University of Patras,\\
26500 Patras,GREECE\\
e-mail: vvlachou@math.upatras.gr

\end{document}